\documentclass[11pt,leqno]{amsart}
\usepackage[letterpaper, total={6in, 8in}, left=1.25in, top=1.5in]{geometry}
\usepackage{amsmath}
\usepackage{amsfonts}
\usepackage{amssymb}
\usepackage{graphicx}
\usepackage{color}
\usepackage{hyperref}
\newtheorem{theorem}{Theorem}[section]
\newtheorem*{theorem*}{Theorem}
\newtheorem{lemma}{Lemma}[section]
\newtheorem{corollary}[theorem]{Corollary}
\newtheorem{proposition}{Proposition}[section]

\def\Ric{\text{Ric}}

\def\a{\alpha}
\def\l{\lambda}

\def\p{\partial}

\def\R{\mathbb{R}}

\def\k{\kappa}

\def\Ric{\operatorname{Ric}}

\def\Sect{\operatorname{Sect}}

\def\diam{\operatorname{diam}}

\numberwithin{equation}{section}

\begin{document}

\title[On the Second Robin eigenvalue of the Laplacian]{On the Second Robin eigenvalue of the Laplacian}

\author{Xiaolong Li}
\address{Department of Mathematics, University of California, Irvine, CA 92697, USA}
\email{xiaolol1@uci.edu}

\author{Kui Wang} 
\address{School of Mathematical Sciences, Soochow University, Suzhou, 215006, China}
\email{kuiwang@suda.edu.cn}

\author{Haotian Wu}
\address{School of Mathematics and Statistics, The University of Sydney, NSW 2006, Australia}
\email{haotian.wu@sydney.edu.au}

\subjclass[2010]{35P15, 49R05, 58C40, 58J50}
\keywords{Second Robin Eigenvalue, Eigenvalue Comparison, First Steklov Eigenvalue}

\begin{abstract}
We study the Robin eigenvalue problem for the Laplace-Beltrami operator on Riemannian manifolds. Our first result is a comparison theorem for the second Robin eigenvalue on geodesic balls in manifolds whose sectional curvatures are bounded from above. Our second result asserts that geodesic balls in nonpositively curved space forms maximize the second Robin eigenvalue among bounded domains of the same volume. 
\end{abstract}

\maketitle




\section{Introduction}

Let $(M^n,g)$ be a complete Riemannian manifold of dimension $n$ and $\Omega \subset M^n$ be a bounded domain with Lipschitz boundary. 
The Robin eigenvalue problem for the Laplace operator on $\Omega$ is
\begin{align}\label{eq 1.1}
\begin{cases} 
-\Delta u =\l u  & \text{ in } \Omega, \\
\frac{\p u}{\p \nu} +\a u  =0 & \text{ on } \p \Omega,
\end{cases}
\end{align}
where $\Delta$ denotes the Laplace-Beltrami operator, $\nu$ denotes the outward unit normal to $\p \Omega$, and $\alpha \in \R$ is the Robin parameter. The eigenvalues, denoted by $\l_{k, \a}(\Omega)$ for $k=1, 2, \ldots$, are increasing and continuous in $\alpha$, and for each $\alpha$ satisfy
\begin{align*} 
\l_{1,\a}(\Omega)<\l_{2,\a}(\Omega) \le \l_{3,\a}(\Omega) \le \cdots \rightarrow \infty,
\end{align*}
where each eigenvalue is repeated according to its multiplicity. The first eigenvalue is simple if $\Omega$ is connected; the first eigenfunction is positive.

The theory of self-adjoint operators yields variational characterizations of the Robin eigenvalues. In particular, the first two eigenvalues are characterized by
\begin{align}\label{eq 1.2}
\l_{1,\a}(\Omega) = \inf \left\{ \frac{\int_{\Omega} |\nabla u|^2 \,d\mu_g +\a \int_{\p \Omega} u^2\,dA_g}{\int_{\Omega}u^2 \,\mu_g} :  u \in W^{1,2}(\Omega)\setminus \{0\} \right\}
\end{align} 
and 
\begin{align}\label{eq 1.3}
\l_{2,\a}(\Omega) = \inf \left\{ \frac{\int_{\Omega} |\nabla u|^2 \,d\mu_g  +\a \int_{\p \Omega} u^2 \,dA_g}{\int_{\Omega}u^2 \,\mu_g} :  u \in W^{1,2}(\Omega)\setminus \{0\},\; \int_{\Omega} u u_1\,d\mu_g =0 \right\}
\end{align}
respectively, where $d\mu_g$ is the Riemnnian measure induced by the metric $g$, $dA_g$ is the induced measure on $\p \Omega$, and $u_1$ is the first eigenfunction associated with $\l_{1,\a}(\Omega)$.

The Robin eigenvalue problem \eqref{eq 1.1} generates a global picture of the spectrum of the Laplace operator. Indeed, the Neumann $(\a =0)$, the Steklov ($\l =0$) and the Dirichlet ($\a \to \infty)$ eigenvalue problems are all special cases of the Robin eigenvalue problem. Hence, existing results on the Dirichlet, Neumann or Steklov eigenvalues naturally motivate the investigation on the Robin eigenvalues.

The classical eigenvalue comparisons of Cheng \cite{Cheng75a} state that the first Dirichlet eigenvalue of a geodesic ball in $(M^n, g)$ with Ricci curvature $\Ric_g\geq(n-1)\kappa g$, where $\k\in\mathbb{R}$, is less than or equal to that of a geodesic ball in a space form of constant sectional curvature $\kappa$, and that the reverse inequality holds if the Ricci lower bound is replaced by the sectional curvature upper bound $\Sect_g\leq\kappa$ and the radius of the geodesic ball is no larger than the injectivity radius at its center. Recently, the first two authors \cite[Theorem 1.1]{LW20} have obtained Cheng type comparison theorems for the first Robin eigenvalue when $\a>0$ and have proved comparison theorems with the reverse inequalities when $\a<0$. Moreover, the comparison theorems hold for the first Robin eigenvalue of the $p$-Laplacian for $p\in(1,\infty)$. 

Motivated by Cheng's eigenvalue comparison theorems, Escobar \cite{Escobar00} proved that in two and three dimensions, the first nonzero Steklov eigenvalue of a geodesic ball, whose radius is assumed to be less than or equal to the injectivity radius at the its center, in a complete manifold with $\Sect_g\leq\kappa$, where $\kappa \in \R$, does not exceed the first nonzero Steklov eigenvalue of a geodesic ball of the same radius in a complete simply connected space form of constant sectional curvature $\kappa$; equality holds if and only if the geodesic balls are isometric. 

Inspired by Escobar's work on the first nonzero Steklov eigenvalue, our first result is a comparison theorem for the second Robin eigenvalue.
\begin{theorem}\label{Cheng}
	Let $(M^n, g)$ be a complete Riemannian manifold with sectional curvature $\operatorname{Sect}_g\le \kappa$ for $\k\in\mathbb{R}$, $B(R)$ be an injective geodesic ball of radius $R$ (i.e., $R$ does not exceed the injectivity radius at its center) in $M^n$, and $B_\k(R)$ be a geodesic ball in an $n$-dimensional complete simply connected space form of constant sectional curvature $\k$. If $n=2$ or $n=3$, and $\a\le 0$, then
	\begin{align*}
	\l_{2,\a}(B(R))\le\l_{2,\a}\left(B_\k(R)\right).
	\end{align*}
	Equality holds if and only if $B(R)$ is isometric to $B_\k(R)$.
\end{theorem}

An analogue of Theorem \ref{Cheng} in higher dimensions holds under additional symmetry assumption, cf. Theorem \ref{Cheng-hd}.

By taking $\a=-\sigma_1(B_\k(R))$, Theorem \ref{Cheng} recovers Escobar's result in \cite{Escobar00}.
\begin{corollary}\label{Cheng-cor}
	Under the hypotheses of Theorem \ref{Cheng}, we have 
	\begin{align*}
	\sigma_1(B(R))\le \sigma_1(B_\k(R)).
	\end{align*}
	Equality holds if and only if $B(R)$ is isometric to $B_\k(R)$.
\end{corollary}

In the second part of this paper, we investigate the shape optimization problem for the Robin eigenvalues.

In Euclidean space, the classical Faber-Krahn inequality asserts that the ball uniquely minimizes the first Dirichlet eigenvalue among bounded domains with the same volume. When $\alpha>0$, the ball is also the unique minimizer of the first Robin eigenvalue $\l_{1,\a}(\Omega)$ among domains of the same volume in $\mathbb{R}^n$, as was shown in dimension two by Bossel \cite{Bossel86} in 1986 and extended to all dimensions $n\geq 2$ by Daners \cite{Daners06} in 2006. An alternative approach via the calculus of variations was found by Bucur and Giacomini \cite{BG10, BG15} later. For negative values of $\a$, it was conjectured by Bareket \cite{Bareket77} in 1977 that the ball would be the maximizer among domains in $\mathbb{R}^n$ with the same volume. However, in 2015, Freitas and Krej\v{c}i\v{r}{i}k \cite{FK15} disproved Bareket's conjecture by showing that the ball is not a maximizer for sufficiently negative values of $\a$. In the same paper, the authors showed that in dimension two, the disk uniquely maximizes $\l_{1,\a}(\Omega)$ for $\a <0$ with $|\a|$ sufficiently small, and conjectured that the maximizer still has radial symmetry whenever $\a<0$ and should switch from a ball to a shell at some critical value of $\a$. 

Let us turn to the shape optimization problem for the second Robin eigenvalue $\l_{2,\a}(\Omega)$. Suppose for the moment that $\Omega\subset\mathbb{R}^n$. When $\a >0$, both the second Dirichlet eigenvalue and the second Robin eigenvalue are uniquely minimized by the disjoint union of two equal balls among bounded Lipschitz domains of the same volume. This was proved by Kennedy \cite{Kennedy09}. When $\a =0$, we have $\l_{2,0}(\Omega)=\mu_1(\Omega)$, the first nonzero Neumann eigenvalue, for which the classical Szeg\"o-Weinberger inequality states that among domains with the same volume, the ball uniquely maximizes $\mu_1(\Omega)$. When $\a<0$, it is expected, cf. \cite[Problem 4.41]{ShapeOptmizationbook}, that $\l_{2,\a}(\Omega)$ should be maximal on the ball for a range of Robin parameters. This expectation has recently been confirmed by Freitas and Laugesen, who proved in \cite{FL18} that the ball uniquely maximizes $\l_{2,\a}(\Omega)$ among domains of the same volume provided that $\a$ lies in a regime connecting the first nonzero Neumann eigenvalue $\mu_1$ and the first nonzero Steklov eigenvalue $\sigma_1$, namely $\a \in [-\frac{n+1}{n}R^{-1}, 0]$, where $R$ is the radius of the ball of the same volume as $\Omega$. Taking $\a=0$ and $\a=-1/R$ recovers the Szeg\"o-Weinberger inequality for $\mu_1(\Omega)$ and the Brock-Weinstock inequality for $\sigma_1(\Omega)$ respectively, and both classical inequalities assert that the ball is the unique maximizer among domains with the same volume in Euclidean space. 

It is well-known that the Faber–Krahn inequality holds in any Riemannian manifold in which the isoperimetric inequality holds, see \cite{Chavel84}. Also, the Sz\"ego-Weinberger inequality holds for domains in the hemisphere and in the hyperbolic space \cite{AB95}. Therefore, it is a natural question to extend the result of Freitas and Laugesen \cite{FL18} to space forms. In this direction, our second result states that in complete simply connected nonpositively curved space forms, geodesic balls uniquely maximize the second Robin eigenvalue among domains with the same volume.

\begin{theorem}\label{thm1}
	Let $(M_\k^n, g_\k)$ be a complete simply connected $n$-dimensional space form of constant sectional curvature $\k$, where $\k\le 0$, and $\Omega\subset M^n_\k$ be a bounded domain with Lipschitz boundary. Let $\Omega^*\subset M^n_\k$ be a geodesic ball of the same volume as $\Omega$, and $\sigma_1(\Omega*)$ be the first nonzero Steklov eigenvalue on $\Omega^*$. If $\a\in[-\sigma_1(\Omega^*),0]$, then
	\begin{align*}
	\l_{2,\a}(\Omega) \le \l_{2,\a}(\Omega^*).
	\end{align*}
	Equality holds if and only if $\Omega$ is a geodesic ball.  
\end{theorem}

When $\k=0$, $\sigma_1(\Omega*)=R^{-1}$, where $R$ is the radius of the ball $\Omega^*$ in $\mathbb{R}^n$. Then Theorem \ref{thm1} says that the ball maximizes the second Robin eigenvalue $\lambda_{2,\a}$, where the Robin parameter $\a\in[-R^{-1},0]$, among domains of the same volume in Euclidean space. In comparison, the same result has been proved in \cite{FL18} for the larger interval, i.e., $[-\frac{n+1}{n}R^{-1}, 0]$, of Robin parameters. We note that these two intervals agree asymptotically as $n\to\infty$. The proof in \cite{FL18} used scaling arguments, which are special to Euclidean space. In contrast, we give a uniform proof for all space forms with nonpositive curvature.

By taking $\a=-\sigma_1(\Omega^*)$, Theorem \ref{thm1} implies that geodesic balls uniquely maximize the first nonzero Steklov eigenvalue among domains of the same volume in complete simply connected nonpositively curved space forms, which recovers a result of Escobar in \cite{Escobar99}. The result in \cite{Escobar99} has recently been generalized to Riemannian manifolds by the authors of this paper in \cite{LWW20}.

This paper is organized as follows. In Section \ref{prelim}, we set up the notation and recall some facts on the eigenfunctions for the second Robin eigenvalue. In Section \ref{cheng-proof}, we prove Theorem \ref{Cheng} and its higher dimensional analogue assuming additional symmetries. Section \ref{thm1-proof} is devoted to the proof of Theorem \ref{thm1}.

\subsection*{Acknowledgements}
We thank Professors Richard Schoen, Lei Ni and Zhou Zhang for their encouragement and support. K. Wang is partially supported by NSFC No.11601359; H. Wu is supported by ARC Grant DE180101348. Both K. Wang and H. Wu acknowledge the excellent work environment provided by the Sydney Mathematical Research Institute.

\section{Preliminaries}\label{prelim}

Throughout the paper, the function $sn_\kappa$ is defined by
\begin{align}\label{sine}
sn_\kappa(t) := 
\begin{cases}
\frac{1}{\sqrt{\kappa}} \sin(\sqrt{\kappa} t), & \text{ if } \kappa >0,\\
t, & \text{ if } \kappa=0,\\
\frac{1}{\sqrt{-\kappa}}\sinh{(\sqrt{-\kappa}t)}, & \text{ if } \kappa<0.
\end{cases}
\end{align}

We fix some notation. For any bounded Lipschitz domain $\Omega\subset M:=M^n$, we denote by $\diam(\Omega)$ the diameter of $\Omega$, by $|\Omega|$ and $|\partial \Omega|$ the $n$-dimensional volume of $\Omega$ and
the $(n-1)$-dimensional Hausdorff measure of $\partial \Omega$ respectively, each taken with respect to the Riemannian metric $g$. Let $(M_\kappa,g_\kappa)$ denote the $n$-dimensional complete simply connected space form of constant sectional curvature $\kappa$ and $\Omega^*_q$ be a geodesic ball in $M_\kappa$ centered at $q$ with $|\Omega^\ast_q|_\kappa=|\Omega|$, where $|\Omega^*_q|_\kappa$ is the $n$-dimensional volume of $\Omega^*$ with respect to $g_\kappa$.

We collect some facts about the Robin eigenfunctions. Denote by $\l_{2,\a}(B_\k(R))$ the second Robin eigenvalue for a geodesic ball $B_\k(R)$ of radius $R$ in the space form $M_\kappa$. The second Robin eigenfunctions are given by 
\begin{align*}
u_i(x)=F(r) \psi_i(\theta),\quad 1\leq i \leq n,
\end{align*}
where $\psi_i(\theta)$'s are the linear coordinate functions restricted to $\mathbb{S}^{n-1}$, and $F(r):[0,R]\to[0,\infty)$ solves the ODE initial value problem
\begin{align}\label{odeF}
F''+(n-1)\frac{sn_\kappa'}{sn_\kappa}F'+\left(\l_{2,\a}(B_\k(R))-\frac{n-1}{sn_\kappa^2}\right)F=0,\quad F(0)=0,\quad F'(0)=1.
\end{align}
We have $F'(R)=-\a F(R)$, and
$\l_{2,\a}(B_\k(R))$ is characterized by
\begin{align}\label{eq 2.2}
\l_{2,\a}(B_\k(R)) =\inf \left\{
\frac{\int_0^R\left( (v')^2+\frac{n-1}{sn_\k^2}v^2\right) sn_\k^{n-1} \, dt +\a  v^2(R)sn_\k^{n-1}(R)}{\int_0^R v^2 sn_\k^{n-1}\,  dt}\right\}
\end{align}
for $ v \in W^{1,2}\left([0,R]\right)\setminus \{0\}$ with $v(0)=0$.

\begin{proposition}\label{prop2.1}
	Let $F(r)$ be the solution to \eqref{odeF}. If $\a<0$, then
	\begin{enumerate}
		\item $F'(r)>0$ for $r\in [0, R]$.
		\item Assume further that $\a\ge-2\frac{sn_\k'(R)}{sn_\k(R)}$. Then $\frac{F'(r)}{F(r)}\ge -\a$ for $r\in (0, R]$.
	\end{enumerate}
\end{proposition}
\begin{proof}
	Let $N(r)=sn^{n-1}_\k(r)F'(r)$, then direct calculation gives
	\begin{align*}
	N'(r)=\left(\frac{n-1}{sn_\k^2(r)}-\l_{2,\a}(B_\k(R))\right)sn_\k^{n-1}(r) F(r),
	\end{align*}
	from which it follows that $N'(r)$ has at most one zero in $(0, R]$ and is positive near $0$. Since that $N(0)=0$ and $N(R)=-\a sn^{n-1}_\k(R)F(R)>0$, we then have 
	$N(r)>0$ for $r\in(0,R]$, proving (1).
	
	To prove (2), let $v(r)=\frac{F'(r)}{F(r)}$. Then $v(R)=-\a$, $v(r)>0$ for $r\in(0,R]$ and $\lim\limits_{r\rightarrow 0^+}v(r)=+\infty$. Rewriting equation \eqref{odeF} as an ODE for $v$ yields
	\begin{align}\label{odev}
	v'+v^2+(n-1)\frac{sn_\k'}{sn_\k}v+\left(\l_{2,\a}(B_\k(R))-\frac{n-1}{sn_\kappa^2}\right)=0.
	\end{align}
	We claim that $v(r)\geq-\a$  on $(0, R]$. Suppose not, then there exists $r_0\in(0,R)$ such that
	\begin{align*}
	v'(r_0)=0,\text{\quad \quad} v''(r_0)\ge 0,\text{\quad and\quad} v(r_0)<-\a.
	\end{align*}
	Then differentiating \eqref{odev} in $r$ and using $sn_k sn_\k'' - (sn_\k')^2=-1$, we have at $r=r_0$ that
	\begin{align}\label{alpha}
	0&=v''(r_0)-(n-1) \frac{v(r_0)}{sn_\k^2(r_0)}+2(n-1)\frac{sn_\k'(r_0)}{sn_\k^3(r_0)} \nonumber\\
	&>(n-1) \frac{\a}{sn_\k^2(r_0)}+2(n-1)\frac{sn_\k'(r_0)}{sn_\k^3(r_0)} \nonumber\\
	&=\frac{n-1}{sn_\k^2(r_0)}\left(\a+2\frac{sn_\k'(r_0)}{sn_\k(r_0)}\right).
	\end{align}
	Again using $sn_k sn_\k'' - (sn_\k')^2=-1$, we see that $sn'(r)/sn(r)$ is monotonically decreasing in $r$, so \eqref{alpha} implies that 
	\begin{align*}
	\a <-2\frac{sn_\k'(R)}{sn_\k(R)},
	\end{align*}
	which contradicts the assumption in (2). Therefore, (2) is proved.
\end{proof}

Recall that the first nonzero Steklov eigenvalue $\sigma_1(\Omega)$ is characterized variationally by
\begin{align}\label{var-stek}
\sigma_1(\Omega)=\inf\left\{ \frac{\int_\Omega |\nabla u|^2 \, d\mu_g}{\int_{\partial \Omega}  u^2 \, dA_g} : u\in W^{1,2}(\Omega)\setminus\{0\},\; \int_{\partial \Omega} u\, dA_g=0 \right\}.
\end{align}

\begin{proposition}\label{prop2.2}
	If $\a\ge -\sigma_1(B_\k(R))$, then $\l_{2,\a}(B_\k(R))\ge 0$.
\end{proposition}
\begin{proof}
	Since 
	\begin{align*}
	\int_{\p B_\k(R)} F(r) \psi_i(\theta) \, d A=0,\quad 1\le i\le n,
	\end{align*}
	the functions $u_i=F(r)\psi_i(\theta)$ ($1\le i\le n$) are test functions for $\sigma_1(B_\k(R))$. Therefore, from \eqref{var-stek} we get
	\begin{align*}
	\sum\limits_{i=1}^n\int_{B_\k(R)}|\nabla u_i|^2 \, d\mu&\ge\sigma_1(B_\k(R))\sum\limits_{i=1}^n\int_{\p B_\k(R)}|u_i|^2 \, dA\\
	&\ge-\a \sum\limits_{i=1}^n\int_{\p B_\k(R)}|u_i|^2 \, dA.
	\end{align*}
	Recall that
	\begin{align*}
	\l_{2,\a}(B_\k(R))&=\frac{\sum\limits_{i=1}^n\int_{B_\k(R)}|\nabla u_i|^2 \, d\mu+\a \sum\limits_{i=1}^n\int_{\p B_\k(R)}|u_i|^2 \, dA}{\sum\limits_{i=1}^n\int_{B_\k(R)}u_i^2 \, d\mu},
	\end{align*}
	so then $\l_{2,\a}(B_\k(R))\ge 0.$
\end{proof}

\section{Comparison theorem for $\lambda_{2,\alpha}$}\label{cheng-proof}
In this section we prove Cheng type comparison theorem for the second Robin eigenvalue.
\begin{proof}[Proof of Theorem \ref{Cheng}]
	Suppose $\a\le 0$. Let $u_1$ be a positive first eigenfunction for $\lambda_{1,\a}(B(R))$ and $\psi_i$ ($1\le i\le n$) be the restriction of the linear coordinate functions on $\mathbb{S}^{n-1}$ ($n\geq2$). Since $\psi_1,\, \psi_2,\ldots,\,\psi_n$ are linearly independent, there exists $\Psi$, a linear combination of $\psi_i$'s, such that
	\begin{align*}
	\int_{ B(R)}F(r) \Psi(\theta) \, u_1(x)\,d\mu_g=0.
	\end{align*}
	So $u(x)=F(r)\Psi(\theta)$ is a test function for $\l_{2,\a}(B(R))$, and hence
	\begin{align*}
	\l_{2,\a}(B(R))\le\frac{\int_{B(R)}|\nabla u|^2 \, d\mu_g+\a \int_{\p B(R)}u^2 dA_g}{\int_{B(R)}u^2 d\mu_g}.
	\end{align*}
	For $n=2$ or $n=3$, the Riemann metric in the (geodesic) polar coordinates has the form $dr^2+f^2(r,\theta) d\theta^2$, where $d\theta^2$ is the standard metric on $\mathbb{S}^{n-1}$. So then
	\begin{align}\label{eq1}
	\int_{B(R)}|\nabla u|^2 \, d\mu_g & =\underbrace{\int_{B(R)}|F'(r)|^2\Psi^2(\theta)\,d\mu_g}_{I}  +\underbrace{\int_{B(R)}\frac{F^2(r)}{f^2(r,\theta)}|\nabla^{\mathbb{S}^{n-1}}\Psi(\theta)|^2\,d\mu_g}_{II} .
	\end{align}
	
	\textbf{Case 1}. Two dimensional case.
	
	When $n=2$, we have $f(r,\theta)=J(r,\theta)$. Computing in the polar coordinates and using integration by parts, we get 
	\begin{align*}
	I&=\int_{B(R)}|F'(r)|^2\Psi^2(\theta) \, d\mu_g\\
	&=\int_{\mathbb{S}^1}\int_0^R |F'(r)|^2\Psi^2(\theta) J(r,\theta) \, dr \,d\theta\\
	&=\int_{\mathbb{S}^1}\Psi^2(\theta)\left(\int_0^R F'(r) J(r,\theta) \, dF(r) \right)\,d\theta\\
	&=\int_{\mathbb{S}^1}\Psi^2(\theta)\left(F'(R)F(R)J(R,\theta)-\int_0^R (F'(r) J(r,\theta))'F(r) \, dr \right)\,d\theta\\
	&=-\a \int_{\mathbb{S}^1}F^2(R)\Psi^2(\theta)J(R,\theta) \, d\theta-\int_{\mathbb{S}^1}\Psi^2(\theta)\left( \underbrace{\int_0^R (F'(r) J(r,\theta))' F(r)\, dr}_{III} \right)\,d\theta,
	\end{align*}
	where we used the boundary condition $F'(R)=-\a F(R)$ in integration by parts. 
	
	Since $\operatorname{Sect}_g\le\kappa$, by the comparison theorem there holds
	\begin{align*}
	\frac{J'(r,\theta)}{J(r,\theta)}\ge \frac{J_\kappa'(r)}{J_\kappa(r)},
	\end{align*}
	which, together with the ODE \eqref{odeF} for $F$, and $F'>0$ on $[0,R]$ proved in Proposition \ref{prop2.1}, implies that
	\begin{align*}
	III &=\int_0^R(F'(r) J(r,\theta))' F(r)\, dr\\
	&=\int_0^R\left(F''(r) +\frac{J'(r,\theta)}{J(r,\theta)}F'(r)\right) F(r)J(r,\theta)\, dr\\
	&\ge\int_0^R\left(F''(r) +\frac{J_\kappa'(r)}{J_\kappa(r)}F'(r)\right) F(r)J(r,\theta)\, dr\\
	&=\int_0^R\left(\frac{ F^2(r)}{J_\kappa^2(r)}-\l_{2,\a}(B_\k(R))F^2(r)\right)J(r,\theta)\, dr\\
	&\ge\int_0^R\frac{ F^2(r)}{J_\kappa(r)}\,dr-\l_{2,\a}(B_\k(R))\int_0^R F^2(r)J(r,\theta)\,dr.
	\end{align*}
	Therefore, we obtain
	\begin{align}\label{eq2}
	I & = \int_{B(R)}|F'(r)|^2\Psi^2(\theta) \, d\mu_g \notag\\
	&\le   -\a\int_{\p B(R)} u^2\,dA_g +\l_{2,\a}(B_\k(R))\int_{ B(R)}u^2 \,d\mu_g -\int_{\mathbb{S}^1}\int_0^R\frac{ F^2(r)}{J_\kappa(r)}\Psi^2(\theta)\,dr \,d\theta.
	\end{align}
	
	Using the comparison result $J(r,\theta)\ge J_\kappa(r)$ and again computing in the polar coordinates, we have 
	\begin{align}
	II&=\int_{B(R)}\frac{F^2(r)}{J^2(r,\theta)}|\nabla^{\mathbb{S}^1}\Psi(\theta)|^2 \, d\mu_g \nonumber\\
	&=\int_{\mathbb{S}^1}\int_0^R\frac{F^2(r)}{J(r,\theta)}|\nabla^{\mathbb{S}^1}\Psi(\theta)|^2\, dr \, d\theta\nonumber\\
	&\le\int_{\mathbb{S}^1}\int_0^R\frac{F^2(r)}{J_\kappa(r)}|\nabla^{\mathbb{S}^1}\Psi(\theta)|^2 \,dr\, d\theta \label{eq3}\\
	&=\int_{\mathbb{S}^1}\int_0^R\frac{F^2(r)}{J_\kappa(r)}\Psi^2(\theta)\,dr\, d\theta. \nonumber
	\end{align}
	where in the last equality we used
	\begin{align}\label{fact}
	\int_{\mathbb{S}^{n-1}}|\nabla^{\mathbb{S}^{n-1}}\Psi(\theta)|^2\,d\theta=(n-1)\int_{\mathbb{S}^{n-1}}\Psi(\theta)^2\,d\theta \quad\text{for}\;\; n\ge 2,
	\end{align}
	which follows from the facts that
	\begin{align*}
	\sum\limits_{i=1}^{n}\psi_i^2=1\quad\text{and}\quad \sum\limits_{i=1}^{n}|\nabla^{\mathbb{S}^{n-1}}\psi_i|^2=n-1.
	\end{align*}
	
	Putting together equality (\ref{eq1}), inequalities (\ref{eq2}) and (\ref{eq3}), we obtain
	\begin{align*}
	\int_{B(R)}|\nabla u|^2 \, d\mu_g\le-\a\int_{\p B(R)} u^2\,dA_g + \l_{2,\a}(B_\k(R))\int_{B(R)}u^2(x)  \, d\mu_g,
	\end{align*}
	thus proving $\l_{2,\a}(B(R))\le\l_{2,\a}(B_\k(R))$.
	
	\textbf{Case 2}. Three dimensional case.
	
	When $n=3$, we have $f(r,\theta)=\sqrt{J(r,\theta)}$.
	By similar calculations as in Case 1, we get 
	\begin{align*}
	I&=\int_{B(R)}|F'(r)|^2\Psi^2(\theta) \, d\mu_g\\
	&=-\a \int_{\mathbb{S}^2}F^2(R)\Psi^2(\theta)J(R,\theta) \,d\theta - \int_{\mathbb{S}^2}\Psi^2(\theta)\left(\underbrace{\int_0^R (F'(r) J(r,\theta))' F(r)\, dr}_{IV} \right)\,d\theta,
	\end{align*}
	where we used the boundary condition $F'(R)=-\a F(R)$ in integration by parts.
	
	Using the comparison results $\frac{J'(r,\theta)}{J(r,\theta)}\ge \frac{J_\kappa'(r)}{J_\kappa(r)}$ and $J(r,\theta)\ge J_\kappa(r)
	$, equation \eqref{odeF} for $F$, and $F'>0$ on $[0,R]$, we obtain
	\begin{align*}
	IV &=\int_0^R(F'(r) J(r,\theta))' F(r)\, dr\\
	&=\int_0^R\left(F''(r) +\frac{J'(r,\theta)}{J(r,\theta)}F'(r)\right) F(r)J(r,\theta)\, dr\\
	&\ge\int_0^R\left(F''(r) +\frac{J_\kappa'(r)}{J_\kappa(r)}F'(r)\right) F(r)J(r,\theta)\, dr\\
	&=\int_0^R\left(2\frac{ F^2(r)}{J_\k(r)}-\l_{2,\a}(B_\k(R))F^2(r)\right)J(r,\theta)\, dr\\
	&\ge\int_0^R2 F^2(r)\, dr-\l_{2,\a}(B_\k(R))\int_0^R F^2(r)J(r,\theta)\,dr.
	\end{align*}
	Therefore, we have
	\begin{align}
	I &=\int_{B(R)}|F'(r)|^2\Psi^2(\theta) \, d\mu_g \nonumber\\
	&\le -\a\int_{\p B(R)} u^2\,dA_g + \l_{2,\a}(B_\k(R))\int_{ B(R)}u^2 \, d\mu_g-2\int_{\mathbb{S}^2}\int_0^RF^2(r)\Psi^2(\theta)\, dr \,d\theta. \label{eq10}
	\end{align}
	
	Direct calculation gives 
	\begin{align}\label{eq11}
	II &= \int_{B(R)}\frac{F^2(r)}{f^2(r,\theta)}|\nabla^{\mathbb{S}^2}\Psi(\theta)|^2 \, d\mu_g \nonumber \\
	&=\int_{\mathbb{S}^2}\int_0^R\frac{F^2(r)}{J(r,\theta)}|\nabla^{\mathbb{S}^2}\Psi(\theta)|^2 J(r,\theta)\, dr \,d\theta\nonumber\\
	&=\int_{\mathbb{S}^2}\int_0^R F^2(r)|\nabla^{\mathbb{S}^2}\Psi(\theta)|^2\, dr \,d\theta\nonumber\\
	&=2\int_{\mathbb{S}^2}\int_0^R F^2(r)\Psi^2(\theta) \,dr\,d\theta,
	\end{align}
	where in the last equality we used \eqref{fact}.
	
	Putting together equalities (\ref{eq1}) and (\ref{eq11}), and inequality (\ref{eq10}, we obtain
	\begin{align*}
	\int_{B(R)}|\nabla u|^2 \, d\mu_g\le-\a\int_{\p B(R)} u^2\,dA_g + \l_{2,\a}(B_\k(R))\int_{B(R)}u^2(x)  \, d\mu_g,
	\end{align*}
	thus implying $\l_{2,\a}(B(R))\le\l_{2,\a}(B_\k(R))$.
	
	The inequalities in the arguments above become equalities if and only if $B(R)$ is isometric to $B_\k(R)$. Therefore, the proof of Theorem \ref{Cheng} is now complete.
\end{proof}

Setting $\a=-\sigma_1(B_\k(R))$, comparison for the second Robin eigenvalue implies comparison for the first nonzero Steklov eigenvalue.

\begin{proof}[Proof of Corollary \ref{Cheng-cor}]
	The first nontrivial Neumann eigenvalue  $\mu_1(B(R))$ is positive. Since $\l_{2,0}(B(R)) = \mu_1(B(R))>0$, $\l_{2,-\sigma_1(B_k(R))}(B(R))\leq \l_{2,-\sigma_1(B_k(R))}(B_\k(R))=0$ by Theorem \ref{Cheng}, and $\l_{2,\a}(B(R))$ is continuous in $\a$, there exists $\a_0\in[-\sigma_1(B_\k(R)),0)$ such that $\l_{2,\a_0}(B(R))=0$. Let $u$ be an eigenfunction for $\l_{2,\a_0}(B(R))$, then the variational characterization \eqref{eq 1.3} implies
	\begin{align*}
	\frac{\int_{B(R)}|\nabla u|^2\,d\mu_g}{\int_{\p B(R)}u^2\,dA_g} & = -\a_0.
	\end{align*}
	On the other hand, for $\lambda_{2,\a_0}(B(R))=0$ and the associated eigenfunction $u$, \eqref{eq 1.1} becomes
	\begin{equation*}
	\begin{cases} 
	\Delta u = 0  & \text{ in } \Omega, \\
	\frac{\p u}{\p \nu} = -\a_0 u & \text{ on } \p \Omega.
	\end{cases}
	\end{equation*}
	So using the divergence theorem, we get
	\begin{align*}
	0 &= \int_{B(R)}\Delta u \, d\mu_g = \int_{\p B(R)}\frac{\p u}{\p \nu} \, dA_g
	= -\a_0 \int_{\p B(R)}u \, dA_g,
	\end{align*}
	which implies that $u$ is a test function for $\sigma_1(B(R))$. By the variational characterization \eqref{var-stek}, we then have
	\begin{align*}
	\sigma_1(B(R))&\le \frac{\int_\Omega |\nabla u|^2 \, d\mu_g}{\int_{\partial \Omega}  u^2 \, dA_g} = -\a_0 \leq \sigma_1(B_\k(R)),
	\end{align*}
	which proves Corollary \ref{Cheng-cor}.
\end{proof}

In dimension four or higher, an analogue of Theorem \ref{Cheng} holds under additional symmetry assumption.

\begin{theorem}\label{Cheng-hd}
	Let $(M^n, g)$, where $n\ge 4$, be a complete Riemannnian manifold with sectional curvature $\operatorname{Sect}_g\le \kappa$ for $\k\in\mathbb{R}$, and $B(R)$ be an injective geodesic ball\footnote{As defined in Theorem \ref{Cheng}.} of radius $R$. Assume that the metric $g$ is centrally symmetric, namely, there is an isometry fixing the center $o$ of $B(R)$ and mapping $\gamma(t)$ to $\gamma(-t)$ for any minimizing geodesic $\gamma$ with $\gamma(0)=o$. If $\a\le 0$, then
	\begin{align*}
	\l_{2,\a}(B(R))\le\l_{2,\a}(B_\k(R)).
	\end{align*}
	Equality holds if and only if $B(R)$ is isometric to $B_\k(R)$.
\end{theorem}
\begin{proof}
	By the symmetry assumption on the metric $g$, we have
	\begin{align*}
	\int_{\mathbb{S}^{n-1}} \psi_i(\theta) J(r,\theta) \, d\theta=0,\quad 1\le i\le n.
	\end{align*}
	So $u_i(x)=F(r)\psi_i(\theta)$ can be used as test functions for $\l_{2,\a}(B(R))$. As a result,
	\begin{align}\label{eq12}
	\l_{2,\a}(B(R))&\le \frac{\sum\limits_{i=1}^n\left(\int_{B(R)}|\nabla u_i|^2 \, d\mu_g+\a \int_{\p B(R)}u_i^2 \,dA_g\right)}{\sum\limits_{i=1}^n\int_{B(R)}u_i^2 \,d\mu_g}\nonumber\\
	&=\frac{\sum\limits_{i=1}^n\int_{B(R)}|\nabla u_i|^2 \, d\mu_g+\a \int_{\p B(R)}F^2 \,dA_g}{\int_{B(R)}F^2\, d\mu_g}.
	\end{align}
	
	Using the Rauch comparison theorem for manifolds with $\Sect_g\leq \kappa$, we estimate
	\begin{align*}
	\sum\limits_{i=1}^n \int_{B(R)}|\nabla u_i|^2 \, d\mu_g&\le \int_{B(R)}|F'(r)|^2+\frac{1}{sn_\k^2(r)} \sum\limits_{i=1}^n|\nabla^{\mathbb{S}^{n-1}}\psi_i(\theta)|^2F^2(r) \, d\mu_g\\
	&=\int_{\mathbb{S}^{n-1}}\int_0^R|F'(r)|^2J(r,\theta)+\frac{(n-1)F^2(r)}{sn_\k^2(r)} J(r,\theta)\, dr \,d\theta.
	\end{align*}
	Recall that
	\begin{align*}
	\int_0^R|F'(r)|^2J(r,\theta)\, dr&\le-\a F^2(R) J(R, \theta)-\int_0^R\frac{ (n-1)F^2(r)}{sn^2_\kappa(r)} J(r,\theta)\, dr\\
	&\quad-\l_{2,\a}(B_\k(R))\int_0^R F^2(r)J(r,\theta)\,dr,
	\end{align*}
	so then
	\begin{align*}
	\sum\limits_{i=1}^n \int_{B(R)}|\nabla u_i|^2 \, d\mu_g&\le -\a F^2(R) \int_{\mathbb{S}^{n-1}}J(R, \theta)\, d\theta+\l_{2,\a}(B_\k(R))\int_{B(R)} F^2(r)\,d\mu_g\\
	&= -\a  \int_{\p B(R)}F^2\, dA_g + \l_{2,\a}(B_\k(R))\int_{B(R)} F^2(r)\,d\mu_g,
	\end{align*}
	Therefore, we conclude from (\ref{eq12}) that 
	\begin{align*}
	\l_{2,\a}(B(R))\le \l_{2,\a}(B_\k(R)).
	\end{align*}
	
	The inequalities in the arguments above are equalities if and only if $B(R)$ is isometric to $B_\k(R)$. Therefore, Theorem \ref{Cheng-hd} is proved.
\end{proof}

In the same way Corollary \ref{Cheng-cor} follows from Theorem \ref{Cheng}, Theorem \ref{Cheng-hd} has the following implication.
\begin{corollary}\label{Cheng-hd-cor}
	Under the hypotheses of Theorem \ref{Cheng-hd}, there holds
	\begin{align*}
	\sigma_1(B(R))\le \sigma_1(B_\k(R)).
	\end{align*}
	Equality holds if and only if $B(R)$ is isometric to $B_\k(R)$.
\end{corollary}

\section{Shape optimization of $\lambda_{2,\alpha}$}\label{thm1-proof}

In this section, we prove that geodesic balls maximize the second Robin eigenvalue among domains with the same volume in nonpositively curved space forms. 

From here on, we assume $\Omega$ to be a bounded domain with Lipschitz boundary in the complete simply connected space form $(M_\k, g_\k)$ with $\Sect_{g_\k}=\k$. Let $\Omega^*\subset M_\k$ be a geodesic ball such that $|\Omega|_\k=|\Omega^*|_\k$. We write $d\mu_{g_\kappa}$ and $dA_{g_\kappa}$ as $d\mu$ and $dA$ respectively for short. 

Let $R$ be the radius of $\Omega^*$. Relabelling the solution to \eqref{odeF} as $F_1$, we define (with a slight abuse of notation) the function $F:[0,\infty)\to[0,\infty)$ by
\begin{equation}\label{defF}
F(r) := 
\begin{cases}
F_1(r), & r\leq R,\\
F_1(R)e^{-\a(r-R)}, & r>R.
\end{cases}
\end{equation}
By definition, $F$ is continuously differentiable on $(0,\infty)$. If $\a\le 0$, then $F$ is non-decreasing on $[0,\infty)$. In below, $\sigma_1(\Omega^*)$ denotes the first nonzero Steklov eigenvalue of $\Omega^*$.

\begin{proposition}\label{H'}
	Assume that $\a \in [-\sigma_1(\Omega^*),0]$. Define $H:[0,\infty)\to\mathbb{R}$ by
	\begin{equation}\label{defH}
	H(r):=(F'(r))^2+\frac{n-1}{sn_\k^2}F^2(r)+2\a F(r)F'(r)+\a\frac{(n-1)sn_\k'(r)}{sn_\k(r)}F^2(r),
	\end{equation}
	where $F(r)$ is defined in \eqref{defF}. Then $H$ is monotonically decreasing on $(0,\infty)$.
\end{proposition}
\begin{proof}
	By assumption, $\a\geq -\sigma_1(\Omega^*)=-\sigma_1(B_\k(R))$, so Proposition \ref{prop2.2} applies and hence $\lambda_{2,\a}(\k,R)\ge 0$. 
	
	We claim that $\sigma_1(\Omega^*)\le 2\frac{sn'_\k}{sn_\k}$. Indeed, by Corollary \ref{Cheng-hd-cor}, $\sigma_1(\Omega^*)\le \sigma_1(B_{\k}(R))$ for $\k\le 0$; in particular, $\sigma_1(\Omega^*)\le\sigma_1(B_0(R)) = R^{-1}$. The claim follows from the elementary inequality $R^{-1} \leq 2\frac{sn'_\k}{sn_\k}$ for $\k\le0$. So part (2) of Proposition \ref{prop2.1} applies and hence $F'\ge-\a F$ on $(0,R]$.
	
	\textbf{Case 1.} $\a \in[-\sigma_1(\Omega_*),0)$.
	
	If $0<r\le R$, then we compute that 
	\begin{align*}
	H'(r)&=\underbrace{2F'F''-\frac{2(n-1)sn_\k'}{sn_\k^3}F^2+\frac{2(n-1)}{sn_\k^2}FF'}_{I}\\
	&\quad +\underbrace{2\a (F')^2+2\a F F''-\a\frac{n-1}{sn_\k^2}F^2+2\a (n-1)\frac{sn_\k'}{sn_\k}FF'}_{II}.
	\end{align*}
	Using the ODE \eqref{odeF} of $F(r)$, we have
	\begin{align*}
	I&=2F'F''-\frac{2(n-1)sn_\k'}{sn_\k^3}F^2+\frac{2(n-1)}{sn_\k^2}FF'\\
	&=2F'\left(-(n-1)\frac{sn'_\k}{sn_\k}F'-\left(\l_{2,\a}(B_\k(R))-\frac{n-1}{sn_\k^2}\right)F\right)-\frac{2(n-1)sn_\k'}{sn_\k^3}F^2\\
	&\quad +\frac{2(n-1)}{sn_\k^2}FF'\\
	&=-\frac{2(n-1)}{sn_\k^3}\left(
	sn_\k^2sn_\k'(F')^2-2sn_\k FF'+sn_\k'F^2
	\right)-2\l_{2,\a}(B_\k(R))FF'\\
	&\le -\frac{2(n-1)}{sn_\k^3}\left(
	sn_\k F'-F\right)^2-2\l_{2,\a}(B_\k(R))FF'\\
	&\le -2\l_{2,\a}(B_\k(R))FF',
	\end{align*}
	where in the first inequality we used $sn'_\k(r)\geq 1$ for $\k\le 0$. Using \eqref{odeF} again,
	\begin{align*}
	II&= 2\a (F')^2+2\a F F''-\a\frac{n-1}{sn_\k^2}F^2+2\a (n-1)\frac{sn_\k'}{sn_\k}FF'\\
	&= 2\a (F')^2+2\a F\left(-(n-1)\frac{sn'_\k}{sn_\k}F'-\left(\l_{2,\a}(B_\k(R))-\frac{n-1}{sn_\k^2}\right)F\right)\\
	&\quad -\a\frac{n-1}{sn_\k^2}F^2+2\a (n-1)\frac{sn_\k'}{sn_\k}FF'\\
	&= 2\a (F')^2+\a\frac{n-1}{sn_\k^2}F^2-2\a\l_{2,\a}(B_\k(R))F^2\\
	&< -2\a\l_{2,\a}(B_\k(R))F^2.
	\end{align*}
	So we conclude
	\begin{align*}
	H'(r)&< -2\l_{2,\a}(B_\k(R))FF'-2\a\l_{2,\a}(B_\k(R))F^2\\
	&\le 2\a\l_{2,\a}(B_\k(R))F^2-2\a\l_{2,\a}(B_\k(R))F^2\\
	&= 0,
	\end{align*}
	where in the second inequality we used $F'\geq -\a F$ on $(0,R]$ and $\lambda_{2,\a}(\k,R)\ge 0$. Therefore, $H(r)$ is monotonically decreasing on $(0,R]$.
	
	If $r\ge R$, then by definition \eqref{defF}, $F(r)=F(R)e^{-\a(r-R)}$. So then
	\begin{align*}
	H(r)=\left(-\a^2+\frac{n-1}{sn_\k^2(r)}+(n-1)\a\frac{sn_\k'(r)}{sn_\k(r)}\right)F^2(r).
	\end{align*}
	Differentiating $H$ in $r$ and using $sn_\k'\geq1$ for $\k\le 0$ and $sn_\k'' sn_\k - (sn_\k')^2=-1$, we have 
	\begin{align*}
	H'(r)&=\left(2\a^3-\frac{2(n-1)sn_\k'(r)}{sn_\k^3(r)}-3\frac{(n-1)\a}{sn_\k^2(r)}-\frac{2(n-1)\a^2}{sn_\k(r)}sn_\k'(r)\right)F^2(r)\\
	&\le\left(2\a^3-\frac{2(n-1)}{sn_\k^3(r)}-3\frac{(n-1)\a}{sn_\k^2(r)}-\frac{2(n-1)\a^2}{sn_\k(r)}\right)F^2(r)\\
	&=-\frac{2(n-1)}{sn_\k^3(r)}\left(1+\frac{3}{2}\a sn_\k(r)+\a^2 sn^2_\k(r)\right)F^2(r) +2\a^3F^2(r)\\
	&=-\frac{2(n-1)}{sn_\k^3(r)}\left(\left(1+\frac{3}{4}\a sn_\k(r) \right)^2+\frac{7}{16}\a^2 sn^2_\k(r)\right)F^2(r) +2\a^3F^2(r)\\
	&<0,
	\end{align*}
	thus proving that $H(r)$ is monotonically decreasing on $[R,\infty)$.
	
	\textbf{Case 2.} $\a=0$.
	
	By the same argument as in Case 1, and using that $\lambda_{2,0}(\k,R)$ is the first nonzero Neumann eigenvalue of $B_\k(R)$, which is positive, we reach the same conclusion that $H'<0$ on $(0,\infty)$.
	
	Therefore, we have proved the proposition.
\end{proof}

We have the following center of mass lemma.
\begin{lemma}\label{lmtest}
	There exists a point $p\in \operatorname{hull}(\Omega)$, the convex hull of $\Omega$, such that
	\begin{equation}\label{test}
	\int_{\Omega} F(r_p(x))\frac{\exp_p^{-1}(x)}{r_p(x)} u_1(x) \,d\mu =0,
	\end{equation}
	where $F$ is defined in \eqref{defF}, $r_p(x)=\operatorname{dist}_g(p,x)$, $\exp_p^{-1}$ is the inverse of the exponential map $\exp_p:T_p M_\kappa\to M_\kappa$, and $u_1$ is a first eigenfunction for $\lambda_{1,\a}(\Omega)$.
\end{lemma}
\begin{proof}
	The proof is similar to \cite[Lemma 4.1]{Edelen17}. 
	Define the vector field
	\begin{align*}
	X(p)=\int_{\Omega} F(r_p(x))\frac{\exp_p^{-1}(x)}{r_p(x)} u_1(x) \, d\mu.
	\end{align*}
	Then the integral curves of $X$ define a mapping from $\operatorname{hull}(\Omega)$ to itself. Since $\operatorname{hull}(\Omega)$ is convex and contained in the injectivity radius, $\operatorname{hull}(\Omega)$, it is a topological ball. Therefore, $X$ must have a zero by the Brouwer fixed point theorem.
\end{proof}

The proof of Theorem \ref{thm1} now proceeds in four propositions.

From here on, we fix the point $p$ according to Lemma \ref{lmtest} so that \eqref{test} holds. Let $(r, \theta)$ denote the polar coordinates centered at $p$ and $J(r, \theta)$ denote the volume element at $(r,\theta)$. Then we have
\begin{align*}
\frac{\exp^{-1}_p(x)}{r_p(x)} = \left(\psi_1(\theta),\psi_2(\theta),\cdots, \psi_n(\theta) \right),
\end{align*}
where $\psi_i$'s are the restrictions of the linear coordinate functions on $\mathbb{S}^{n-1}$. We define
\begin{align*}
v_i(x) := F(r_p(x)) \psi_i(\theta), \quad 1\leq i\leq n,
\end{align*}
and rewrite \eqref{test} as
\begin{align*}
\int_{\Omega} v_i(x)u_1(x) \, d\mu =0, \quad 1\leq i\leq n.
\end{align*}
So $v_i$'s are test functions for $\lambda_{2,\a}(\Omega)$.
\begin{proposition}\label{prop 4.2}
	Under the hypotheses of Theorem \ref{thm1}, there holds
	\begin{equation}\label{sig11}
	\l_{2,\a}(\Omega) \le \frac{\int_{\Omega} \left|F'(r_p)\right|^2+\frac{n-1}{sn_\kappa^2(r_p)}F^2(r_p) \, d\mu+\a\int_{\p \Omega} F^2(r_p)\, dA}{\int_{\Omega}  F^2(r_p)\, d\mu}.
	\end{equation}
\end{proposition}

\begin{proof}
	We denote by $\nabla^{\mathbb{S}^{n-1}}$ the covariant derivative  with respect to the standard metric on $\mathbb{S}^{n-1}$, and by $\nabla$ the covariant derivative with respect to the metric $g_\k=dr^2+sn^2_\k(r) d\theta^2 $ on $M_\k$. Using 
	\begin{align*}
	\sum_{i=1}^n\psi^2_i=1 \quad\text{and} \quad \sum_{i=1}^n |\nabla^{\mathbb{S}^{n-1}} \psi_i|^2=n-1,
	\end{align*}
	we compute that
	\begin{align}\label{eq 4.3}
	\sum_{i=1}^n \int_{\Omega} \left|\nabla  v_i\right|^2 \, d\mu \nonumber
	&= \sum_{i=1}^n \int_{\Omega} \left|\nabla  \left(F(r_p)\psi_i\right)\right|^2\, d\mu \nonumber \\
	&=\sum_{i=1}^n\int_{\Omega}\left( \left|F'(r_p)\right|^2  \psi^2_i + \frac{F^2(r_p)}{sn^2_\k(r_p)} |\nabla^{\mathbb{S}^{n-1}} \psi_i|^2 \right) \, d\mu \nonumber \\
	&= \int_{\Omega} \left( \left|F'(r_p)\right|^2+\frac{n-1}{sn_\kappa^2(r_p)}F^2(r_p)\right) \, d\mu.
	\end{align}
	On the other hand, 
	\begin{equation}\label{eq 4.4}
	\sum_{i=1}^n \int_{\p \Omega} v_i^2 \, dA =\sum_{i=1}^n \int_{\p \Omega} |F(r_p)|^2 \psi_i^2\, dA  =\int_{\p \Omega} |F(r_p)|^2\, dA.
	\end{equation}
	So using the averaging of Rayleigh quotients for $v_i$'s, and \eqref{eq 4.3} and \eqref{eq 4.4}, we obtain
	\begin{align*}
	\l_{2,\a}(\Omega) &\le  \frac{\sum\limits_{i=1}^n \int_{\Omega} |\nabla v_i|^2\, d\mu+\a\sum\limits_{i=1}^n  \int_{\p \Omega} v_i^2\, dA}{\sum\limits_{i=1}^n  \int_{\Omega} v_i^2\, d\mu} \\
	&= \frac{\int_{\Omega} \left|F'(r_p)\right|^2+\frac{n-1}{sn_\kappa^2(r_p)}F^2(r_p) \, d\mu+\a\int_{\p \Omega} F^2(r_p)\, dA}{\int_{\Omega}  F^2(r_p)\, d\mu}.
	\end{align*}
	This proves the proposition. 
\end{proof}

\begin{proposition}\label{prop 4.3}
	Under the hypotheses of Theorem \ref{thm1}, there holds
	\begin{align}
	\l_{2,\a}(\Omega)\le \frac{\int_{\Omega} H(r_p) \, d\mu}{\int_\Omega  F^2(r_p) \, d\mu}.\label{20}
	\end{align}
\end{proposition}
\begin{proof}
	Using $|\nabla r_p | =1$, we have
	\begin{align}\label{sig2}
	\int_{\partial \Omega} F^2(r_p) \, dA \nonumber &\ge \int_{\partial \Omega} F^2(r_p)\langle \nabla r_p, \nu \rangle \, dA \nonumber\\
	&=\int_\Omega \operatorname{div}\left(F^2(r_p)\nabla r_p\right)\, d\mu \nonumber\\
	&=\int_\Omega \left( (F^2)'(r_p) + F^2(r_p)\Delta r_p \right) \, d\mu \nonumber\\
	&=\int_\Omega \left( (F^2)'(r_p) +\frac{(n-1)sn_\kappa'(r)}{sn_\kappa(r)}F^2(r_p) \right) \, d\mu.
	\end{align} 
	Substituting \eqref{sig2} into \eqref{sig11} and recalling definition \eqref{defH} of $H$, then \eqref{20} follows.
\end{proof}

Let $\Omega_p^*$ be the geodesic ball having volume $|\Omega|_\k$ and centered at $p$ so that \eqref{test} holds.

\begin{proposition}\label{movemass}
	Under the hypotheses of Theorem \ref{thm1}, we have
	\begin{align}\label{eq:movemass}
	\frac{\int_{\Omega} H(r_p)\ d \mu}{\int_{\Omega} F^2(r_p)\, d\mu} \le \frac{\int_{\Omega_p^*} H(r_p)\ d \mu}{\int_{\Omega_p^*} F^2(r_p)\, d\mu}.
	\end{align}
	Equality holds if and only if $\Omega=\Omega^*_p$.
\end{proposition}

\begin{proof}
	Recall that $F$ defined in \eqref{defF} is non-decreasing, we have
	\begin{align}\label{ineqF^2}
	\int_{\Omega} F^2(r_p) d\mu &= \int_{\Omega\cap\Omega_p^*} F^2(r_p) d\mu + \int_{\Omega\setminus\Omega_p^*} F^2(r_p) d\mu \nonumber \\
	&\ge \int_{\Omega\cap\Omega_p^*} F^2(r_p) d\mu + \int_{\Omega\setminus\Omega_p^*} F^2(R) d\mu \nonumber \\
	&\ge \int_{\Omega_p^*} F^2(r_p) d\mu.
	\end{align}
	By Proposition \ref{H'}, $H$ is monotonically decreasing, so then
	\begin{align}\label{ineqH}
	\int_{\Omega} H(r_p) d\mu &= \int_{\Omega\cap\Omega_p^*} H(r_p) d\mu + \int_{\Omega\setminus\Omega_p^*} H(r_p) d\mu \nonumber \\
	&\le \int_{\Omega\cap\Omega_p^*} H(r_p) d\mu + \int_{\Omega\setminus\Omega_p^*} H(R) d\mu \nonumber \\
	&\le \int_{\Omega_p^*} H(r_p) d\mu.
	\end{align}
	Inequality \eqref{eq:movemass} follows from \eqref{ineqF^2} and \eqref{ineqH}. In particular, equality in \eqref{eq:movemass} holds if and only if both \eqref{ineqF^2} and \eqref{ineqH} are equalities, which occurs if and only if $\Omega=\Omega^*_p$.
	
	Therefore, the proposition is proved.
\end{proof}

\begin{proposition}\label{propsigma2}
	Under the hypotheses of Theorem \ref{thm1}, there holds
	\begin{align*}
	\l_{2,\a}(\Omega_p^*) =\frac{\int_{\Omega_p^*} H(r_p)\, d\mu}{\int_{\Omega_p^*} F^2(r_p)\, d\mu} .
	\end{align*}
\end{proposition}
\begin{proof}
	Recall that $F(r)\psi_i(\theta)$ are the eigenfunctions corresponding to $\l_{2,\a}(\Omega_p^*)$, so then
	\begin{align*}
	\l_{2,\a}(\Omega_p^*)=\frac{\int_{\Omega_p^*} |F'|^2(r_p)+\frac{n-1}{sn_\kappa^2 (r_p)}F^2(r_p)\, d\mu + \a\int_{\partial \Omega_p^*} F^2(r_p)\, dA}{\int_{\Omega_p^*} F^2(r_p)\, d\mu}
	\end{align*}
	and
	\begin{align*}
	\int_{\partial \Omega_p^*} F^2(r_p)\, dA&=\int_{\partial \Omega_p^*} \langle F^2(r_p)\nabla r_p, \nu\rangle\, dA\\
	&=\int_{\Omega_p^*} \operatorname{div}\left( F^2(r_p)\nabla r_p \right)\, d\mu\\
	&=\int_{\Omega_p^*}\left( (F^2)'+F^2\Delta r_p \right) \, d\mu\\
	&=\int_{\Omega_p^*}\left( (F^2)'+\frac{(n-1)sn'_\kappa}{sn_\kappa}F^2 \right)\, d\mu.
	\end{align*}
	
	Therefore, we have proved the proposition.
\end{proof}

\begin{proof}[Proof of Theorem \ref{thm1}]
	The theorem follows from combining Propositions \ref{prop 4.3}--\ref{propsigma2}. 
\end{proof}


\end{document}